\documentclass[10pt]{article}

\usepackage{amsmath}
\usepackage{amsfonts}
\usepackage{color}
\usepackage{amssymb}
\usepackage{graphicx}
\usepackage{enumerate}
\usepackage[all]{xy}

 \allowdisplaybreaks

\begin{document}

\newtheorem{problem}{Problem}

\newtheorem{theorem}{Theorem}[section]
\newtheorem{corollary}[theorem]{Corollary}
\newtheorem{definition}[theorem]{Definition}
\newtheorem{conjecture}[theorem]{Conjecture}
\newtheorem{question}[theorem]{Question}
\newtheorem{lemma}[theorem]{Lemma}
\newtheorem{proposition}[theorem]{Proposition}
\newtheorem{quest}[theorem]{Question}
\newtheorem{example}[theorem]{Example}
\newenvironment{proof}{\noindent {\bf
Proof.}}{\rule{2mm}{2mm}\par\medskip}
\newenvironment{proofof}{\noindent {\bf
Proof of the Theorem 6.1.}}{\rule{2mm}{2mm}\par\medskip}
\newcommand{\remark}{\medskip\par\noindent {\bf Remark.~~}}
\newcommand{\pp}{{\it p.}}
\newcommand{\de}{\em}

\title{  {The $Q$-index and connectivity of graphs}\thanks{ This work is supported by the National Natural Science Foundation of China (Nos. 11971311, 
12026230); L. Feng and W. Liu were supported by NSFC (Nos. 11871479, 12071484), Hunan Provincial Natural Science Foundation (2020JJ4675, 2018JJ2479).
E-mail addresses: zpengli@sjtu.edu.cn (P.-L. Zhang), fenglh@163.com (L. Feng),  wjliu6210@126.com(W. Liu), xiaodong@sjtu.edu.cn ($^\dag$X.-D. Zhang, corresponding author).}}

\author{
 Peng-Li Zhang$^a$,  Lihua Feng$^b$, Weijun Liu$^b$, Xiao-Dong Zhang$^a\dagger$ \\
{\small $^a$ School of Mathematical Sciences, Shanghai Jiao Tong University,}\\ {\small 800 Dongchuan Road, Shanghai, 200240, PR China} \\
{\small $^b$ School of Mathematics and Statistics, Central South University,}\\ {\small  New Campus,  Changsha, Hunan, 410083, PR China} \\
}

\date{}

\maketitle

\vspace{-0.5cm}

\begin{abstract}
A connected graph $G$ is said to be $k$-connected  if it has more than $k$ vertices and remains connected whenever fewer than $k$ vertices are deleted. In this paper,  for a connected graph $G$ with sufficiently large order, we present a tight sufficient condition for $G$ with fixed minimum degree to be  $k$-connected based on  the $Q$-index. Our result can be viewed as a spectral counterpart of the corresponding Dirac type condition.
 \end{abstract}

{{\bf Key words:}   $Q$-index; Minimum degree; $k$-connected.

 }

\section{Introduction}
 All graphs considered in this paper are simple connected and undirected. The notations we used are standard. Let $G$ be a simple connected  graph with vertex set $V(G)$ and edge set $E(G)$ such that $|V(G)|=n$ and $|E(G)|=m$. Let $d(v)$ be the degree of a vertex $v$ in $G$, and  the minimum degree  be $\delta(G)=\delta$. For two vertex-disjoint graphs $G$ and $H,$  we denote $G\cup H$ the disjoint union of $G$ and $H,$ $G\vee H$ the join of  $G$ and $H,$ which is a graph obtained by adding all possible edges between $G$ and $H.$  Throughout this paper, we use the symbol $i\sim j$ to denote the vertices $i$ and $j$ are adjacent, and $i\nsim j$ otherwise.



A graph $G$ is said to be \emph{$k$-connected} if it has more than $k$ vertices and remains connected whenever fewer than $k$ vertices are deleted.
In the meantime, a vertex-cut $X$ of $G$ is a subset of $V(G)$ such that $G-X$ is disconnected. The vertex connectivity $\kappa$ is the minimum vertex-cut $X.$ We say  $G$ is $k$-connected when $\kappa \geq k,$ $\kappa=0$ if $G$ is either trivial or disconnected. In other words, $G$ is $k$-connected if the minimum vertex-cut $X$ satisfies $|X|\geq k.$



The \emph{adjacency matrix} of $G$ is $A(G)=(a_{ij})_{n \times n}$ with $a_{ij}=1$ if  $i$ and $j$ are adjacent, and $a_{ij}=0$ otherwise.   The largest eigenvalue of $A(G)$, denoted by $\lambda (G)$, is called the {\it spectral radius} of $G$. The \emph{diagonal matrix} of $G$ is $D(G)=(d_{ii})_{n \times n},$ whose diagonal entries $d_{ii}$ satisfy $d_{ii}=d(i)$. The \emph{signless Laplacian matrix} $Q(G)$ of $G$ is defined as $D(G)+A(G).$ The largest eigenvalue of $Q(G)$, denoted by $q (G)$, is called the  $Q$-\textit{index} (or the signless Laplacian  spectral radius) of $G$.

When one talks about spectral graph theory, perhaps  one of the most well-known problems  is the
Brualdi-Solheid problem \cite{brualdi}:  Given a set ${\cal
{G}}$ of graphs, find a tight upper bound for the spectral radius in
${\cal
{G}}$ and characterize the extremal graphs.
This problem is well studied in the
literature for various classes of graphs, such as graphs with given number of cut vertices or cut edges  \cite {Belardo, Liuruifang},  graphs with given edge chromatic number \cite{FengLAA16}.
For the  $Q$-index counterpart of the above problem, Zhang \cite{Zhangxiaodong} gave  the  $Q$-index of graphs with given degree sequence, Zhou \cite{Zhoubo} studied the $Q$-index and Hamiltonicity.  Also, from both theoretical and practical viewpoint,
the  eigenvalues of graphs have been successfully used in many   other disciplines,  one may
refer to \cite{Huobofeng, LiShiEnergy, Lixueliang13, ZhangMinjieDAM, ZhangMinjieAMC}.

Analogous to the Brualdi--Solheid problem, the following    problem
was proposed \cite{NikiforovLAA10}:
What is the maximum spectral radius of a graph $G$ on $n$
vertices without a subgraph isomorphic to a given graph $F$?
Regarding this problem,
 Fiedler and Nikiforov  \cite{FiedlerNikif} obtained tight sufficient conditions for
graphs to be hamiltonian or traceable. This motivates   further study for such questions. Later, Zhou \cite{Zhoubo} considered  the $Q$-index version of the results in \cite{FiedlerNikif}. For further reading in this topic, see   \cite{FengLAA17, FengMonoshMath, FengDAM17,   Liyawen, Liu,  LiuDMGT, Lumei,    Ningbo15,   ZhouWangligong2, ZhouWangligongAMC18}.

 For the connectivity and eigenvalues of graphs, one must mention the classical result from Fiedler \cite{Fiedler73} which states that the second smallest Laplacian eigenvalue is at most the connectivity for any non-complete graph, which now becomes one of the most attractive research areas. For adjacency eigenvalues,  extending  the  result in \cite{Chandran}, Cioab\v a \cite{Cioaba10LAA} obtained
 \begin{theorem}
 Let $d\geq k\geq 2$. If the second largest eigenvalue $\lambda_2$ of a $d$-regular graph satisfies
$$
\lambda_2<d -\frac{ (k-1)n}{(d+1)(n-d-1)},
$$
then the  edge-connectivity of $G$ is at least $k$.
 \end{theorem}
There are also several related results regarding the edge-connectivity and eigenvalues of graphs, which can be found in \cite{GuxiaofengJGT, LiuhuiqingLumei, OSuilSebiSIAM}.


One of the classical problems of graph theory is to obtain sufficient conditions for a graph possessing certain properties. It is known that \cite[Page 4]{Bollobas78}, if $G$ is a simple graph of order $n\geq k+1$, and  if
$$
\delta\geq \frac12(n+k-2),
$$
then $G$ is  $k$-connected. In this paper,
borrowing ideas from \cite{ Liyawen, Nikiforov},  by utilizing   the  $Q$-index, we will establish a new   sufficient   condition  for  graphs with fixed minimum degree to be  $k$-connected, for sufficiently large order (and therefore for relatively small $\delta$).  Such results may be of independent interest.
For any $k>1$ and $n>2k+1,$  let
$${M}_k(n)=K_k\vee(K_{n-2k}\cup \overline{K_k}).$$
For any $k\geq 1$ and $n\geq k+2,$  let
$${L}_k(n)=K_1\vee(K_{n-k-1}\cup  {K_k}).$$
In light of the result of Li and Ning \cite{Libinglong}, Nikiforov \cite{Nikiforov} proved the following theorem.
\begin{theorem}\label{th12}
Let $k>1,$ $n\geq k^3+k+4,$ and let $G$ be a graph of order  $n$  with minimum degree $\delta(G)  \geq k$.
  If
  $$
  \lambda(G)\geq  n-k-1,
  $$
then $G$ has a Hamiltonian cycle unless $G={M}_k(n)$ or $G={L}_k(n).$
\end{theorem}

We define
\begin{eqnarray*}
 \mathcal{M}_1(n,k)&=&\left\{G\subseteq {M}_k(n)-E', \mbox{where $ E'\subset E_1({M}_k(n))$ with $|E'|\leq \lfloor\frac{k^2}{4}\rfloor $}\right\},\\
 \mathcal{L}_1(n,k)&=&\left\{G\subseteq {L}_k(n)-E', \mbox{where $ E'\subset E_1({L}_k(n))$ with $|E'|\leq \lfloor\frac{k^2}{4}\rfloor$}\right\}.
\end{eqnarray*}
 Li, Liu and Peng \cite{Liyawen} recently obtained the   $Q$-index counterpart of Theorem \ref{th12}.
\begin{theorem}\label{th13}
Let $k>1,$ $n\geq k^4+k^3+4k^2+k+6.$ Let $G$ be a connected graph with  $n$ vertices and minimum degree $\delta(G)  \geq k$.
  If
  $$
  q(G)\geq  2(n-k-1),
  $$
then $G$ has a Hamilton cycle unless $G\in \mathcal{M}_1(n,k)$ or $G\in \mathcal{L}_1(n,k).$
\end{theorem}

For convenience, for the rest of this paper, we denote
$$
{A}(n,k,\delta):=K_{k-1}\vee(K_{\delta-k+2}\cup  {K_{n-\delta-1}}).
$$
Obviously, for any   integers $k>1,\delta\geq 1$ and $n> \delta+1,$  ${A}(n,k,\delta)$ is not $k$-connected.
For the graph ${A}(n,k,\delta),$ let
 $$X:=\{v \in V({A}(n,k,\delta)):d(v)=\delta\}, \qquad Y:=\{v \in V({A}(n,k,\delta)):d(v)=n-1 \},$$
 $$ Z:=\{v \in V({A}(n,k,\delta)):d(v)=n-\delta+k-3\}.$$
Let $E'$ denote the edge set of $E({A}(n,k,\delta))$ whose endpoints are both from $Y\cup Z.$
 We define
$$ \mathcal{A}_1(n,k,\delta):\\
=\left\{G\subseteq {A}(n,k,\delta)-E', \mbox{where $ E'\subset E({A}(n,k,\delta))$ with $|E'|\leq \left\lfloor \frac{ (\delta-k+2)(k-1)}{4}\right\rfloor$}\right\},$$
$$ \mathcal{A}_2(n,k,\delta):=\left\{G\subseteq {A}(n,k,\delta)-E', \mbox{where $ E'\subset E({A}(n,k,\delta))$ with $|E'|= \left\lfloor \frac{ (\delta-k+2)(k-1)}{4}\right\rfloor+1 $}\right\},$$
$$F(k,\delta):=  (k^2+2k-3)\delta^2- (2k^3-k^2-17k+8)\delta+ k^4-3k^3-8k^2+23k+4.$$

In \cite{FengMonoshMath}, using the adjacency spectral radius, it is obtained that
 \begin{theorem}\label{th14}
Let $\delta \geq k \geq 3, n\geq (\delta-k+2)(k^2-2k+4)+3.$ Let $G$ be a connected graph of order  $n$  and  minimum degree $\delta(G) \geq \delta $.
  If
  $$ \lambda(G)\geq n-\delta+k-3,$$
then $G$ is $k$-connected unless $G= {A}(n,k,\delta)$.
\end{theorem}

 Motivated by Theorem \ref{th13}, as the $Q$-spectral counterpart of Theorem \ref{th14}, we have the following main result of this paper.
\begin{theorem}\label{th15}
Let $G$ be a connected graph of order  $n$  with  minimum degree $\delta(G) = \delta \geq k \geq 3.$
  If $n\geq  F(k,\delta) $and
  $$ q(G)\geq 2(n-\delta+k-3),$$
then $G$ is $k$-connected unless $G\in \mathcal{A}_1(n,k,\delta)$.
\end{theorem}
%

\section{Preliminaries}
In this section, we present some basic notations and lemmas.

Let $x=(x_{1},x_{2},...,x_{n})^T\neq 0,$ by Rayleigh's principle, we have
 $$ q(G)=\max_{x}\frac{\langle Q(G)x,x \rangle}{\langle x,x \rangle}=\max_{x}\frac{ x^{T}Q(G)x }{x^{T}x}.$$
By the definition of $Q(G),$ we   have
$$ \langle Q(G)x,x\rangle= \sum_{i\sim j}(x_{i}+x_{j})^2.$$
If $z$ is the corresponding unit positive eigenvector (usually called Perron vector) of $q(G),$ then
$$
Q(G)z=q(G)z.
$$
According to the Perron-Frobenius theorem,
we have $x_{i}>0$ for each
$i \in V(G)$ if $G$ is connected.
Taking the $i$-th entry of both sides and rearranging terms, we have
\begin{eqnarray} \label{equaion1}
(q(G)-d(i))z_{i}=\sum_{i\sim j}z_{j}.
\end{eqnarray}
Let $N(i)$ denote the set of neighbours of $i,$ $N[i]=N(i)\cup \{i\}$. From the above,  we have the following lemma.
\begin{lemma}\label{lemma21}\cite{Liyawen}
For any $i,j\in V(G),$ we have
\begin{eqnarray} \label{equaion2}
   (q(G)-d(i))(z_{i}-z_{j})=(d(i)-d(j))z_{j}+\sum_{k\in N(i)\setminus N(j)}z_{k}-\sum_{l\in N(j)\setminus N(i)}z_{l}.
\end{eqnarray}
\end{lemma}

\begin{lemma}\cite{Fenglihua} \label{lemma22}
Let $G$ be a graph of order $n.$ Then
  $$ q(G)\leq \frac{2m}{n-1}+n-2.$$
\end{lemma}

Using the ideas in \cite{Hongzhenmu}, we obtain
\begin{lemma}\label{lemma23}
Let $G$ be a connected graph of order $n\geq 2\delta-k+5,$ size $m,$ minimum degree $\delta(G) =\delta \geq k \geq 2.$ If
  $$ m > \frac{1}{2}n(n-1)-(\delta-k+3)(n-\delta-2),$$
then $G$ is $k$-connected unless $G$ is a subgraph of ${A}(n,k,\delta).$
\end{lemma}

\begin{proof}
 Suppose  on the contrary that $G$ is not $k$-connected. Let $X$ be a minimum vertex-cut with $1\leq |X| \leq k-1.$  Assume that  $C_1,C_2,...,C_t$ $(t>1)$ are the components of $G-X,$ where $|C_1|\leq |C_2| \leq...\leq |C_t|. $
Clearly,  for  $1\leq i \leq t,$  each vertex in $C_i$  is adjacent to at most $|C_i|-1$ vertices of $C_i$ and $|X|$ vertices of $X.$ Thus
$$\delta|C_i|\leq \sum _{x\in C_i}d(x)\leq (|C_i|-1+|X|)|C_i|,$$
hence  $|C_i|\geq \delta-|X|+1,$ and therefore $\delta-|X|+1 \leq |C_i| \leq n-|X|-(\delta-|X|+1),$ which implies
\begin{eqnarray*} \label{equaion3}
\delta-|X|+1\leq |C_i|\leq n-\delta-1.
\end{eqnarray*}
Let $ S=\cup_{i=2}^{t}C_i.$ Then from above, $\delta-|X|+1\leq |S|\leq n-\delta-1.$ Since $G-X$ is disconnected,  there are no edges between $C_1$ and $S$ in $G,$ we obtain
$$m \leq \frac{1}{2}n(n-1)-|C_1||S|.$$

%

In order to prove  that $G$ is a subgraph of ${A}(n,k,\delta),$ it suffices to show that $|C_1|=\delta-k+2.$

If $|C_1|\geq \delta-k+3,$  as $|C_1|\leq |C_2| \leq...\leq |C_t|$ and $S=\cup_{i=2}^{t}C_i,$  we   have $|C_1| \leq \frac{n-|X|}{2}.$  Therefore for $\delta-k+3 \leq |C_1| \leq \frac{n-|X|}{2},$   we have   $|C_1||S|=|C_1|(n-|X|-|C_1|) \geq (\delta-k+3)(n-|X|-(\delta-k+3))$, the equality  is attained when $|C_1|=\delta-k+3.$  Since $|X|\leq k-1,$
\begin{eqnarray*}
 m &\leq& \frac{1}{2}n(n-1)-|C_1||S| \\
&\leq& \frac{1}{2}n(n-1)-(\delta-k+3)(n-|X|-(\delta-k+3)) \\
&\leq& \frac{1}{2}n(n-1)-(\delta-k+3)(n-\delta-2).
\end{eqnarray*}
From the assumption,
we get a contradiction. Thus  $|C_1|\leq \delta -k +2.$
Combining this with $|C_1|\geq \delta -|X| +1 \geq\delta -k +2,$ we have $|C_1|=\delta-k+2.$

Hence, as the minimum degree of $G$ is $\delta$, we have $|X|=k-1$
and $d_G(i)= \delta $ for each $i\in C_{1},$ therefore each vertex of $C_1$ is adjacent to each vertex of $X.$  We obtain the result.
\end{proof}
\section{Proof of the Main Result}
To prove Theorem \ref{th15},  we still need to prove the following  several lemmas.

\begin{lemma}\label{lemma31}
Assume  $\delta\geq k\geq 3,$ let $G$ be a connected graph of order $n\geq F(k,\delta)$ and  minimum degree $\delta \geq k.$
For each graph $G\in \mathcal{A}_1(n,k,\delta),$
we have $q(G)\geq 2(n-\delta+k-3).$
\end{lemma}
\begin{proof}
Let $G$ be a graph in $\mathcal{A}_1(n,k,\delta).$  We easily get that $q(K_{n-\delta+k-2}\cup \overline{K_{\delta-k+2}})=2(n-\delta+k-3)$. Now we construct a vector $z,$ where $z_i=1$ if $i\in Y\cup Z,$ $z_j=0$ if $j\in X.$ Obviously $z$ is the corresponding eigenvector to $q(K_{n-\delta+k-2}\cup \overline{K_{\delta-k+2}}).$ Then we obtain
$$\langle Q(G)z,z \rangle-\langle Q(K_{n-\delta+k-2}\cup \overline{K_{\delta-k+2}})z,z\rangle=(\delta-k+2)(k-1)-4|E'|\geq 0.$$
By the Rayleigh's principle, we have $q(G)\geq 2(n-\delta+k-3).$
\end{proof}

\begin{lemma}\label{lemma32}
Assume  $\delta\geq k\geq 3,$ let $G$ be a connected graph of order $n\geq F(k,\delta)$ and  minimum degree $\delta \geq k.$
For each graph $G\in \mathcal{A}_2(n,k,\delta),$
we have $q(G)> 2(n-\delta+k-3)-1.$
\end{lemma}
\begin{proof}
 Let $z$ be the  vector defined in Lemma \ref{lemma31}. We have
$$\langle Q(G)z,z \rangle-\langle Q(K_{n-\delta+k-2}\cup \overline{K_{\delta-k+2}})z,z\rangle=(\delta-k+2)(k-1)-4|E'|\geq -4.$$
Similarly, we have $q(G)\geq 2(n-\delta+k-3)-\frac{4}{\|z\|^2}>2(n-\delta+k-3)-1.$
\end{proof}

We put our attention to prove $q(G)<2(n-\delta+k-3)$ for  $G\in \mathcal{A}_2(n,k,\delta)$ in the following.

 Let $G$ be a graph among $\mathcal{A}_2(n,k,\delta)$ with the largest $Q$-index, assume further that the induced subgraph $G[Y]$ contains the largest number of edges.  Then $|Y|=k-1\geq 2,$ i.e., $k\geq 3.$

In the following, let $x$ be the eigenvector corresponding to $q(G).$
Moreover we may assume $\max\limits_{i\in V(G)}x_i=1.$
Following this, we  have
\begin{lemma}\label{lemma33}
Assume $G\in \mathcal{A}_2(n,k,\delta)$ as defined above. For each $i\in X,$ we have $$x_i\leq \frac{k-1}{q(G)-(2\delta-k+1)}.$$
\end{lemma}
\begin{proof}
Using equation (\ref{equaion1}) at vertex  $i,$ we have
$$
(q(G)-d(i))x_i=\sum_{j\in X\setminus i}x_j+\sum_{j \in Y}x_j
.
$$
As $d(i)=\delta,$ $x_i$ is the same for all vertex in $X,$  and $\max\limits_{i\in V(G)}x_i=1$,  we have
$$\left(q(G)-\big(\delta+(\delta-k+1)\big)\right)x_i=\sum_{j \in Y}x_j.$$
The proof is completed.
\end{proof}

Now we divide $Y,Z$ into the following two parts, respectively.
$$Y_1=\{i\in Y: d(i)=n-1\},\quad Y_2=\{i\in Y: d(i)\leq n-2\},$$
$$Z_1=\{i\in Z: d(i)=n-\delta+k-3\},\quad Z_2=\{i\in Z: d(i)\leq n-\delta+k-4\}.$$

We first declare the following truth: $Z_1\neq \emptyset $ since $ n-\delta-1 >2\left(\left\lfloor \frac{ (\delta-k+2)(k-1)}{4}\right\rfloor+1\right)+1,$ and $n\geq F(k,\delta).$

We already know the upper bound of $x_i$ for each $i\in X$ and $x_i<1.$ Clearly   $\max\limits_{i\in V(G)}x_{i}=\max\limits_{i\in Y\cup Z}x_{i}.$

We also need   the following several lemmas.

\begin{lemma}\label{lemma34}
Let $G\in \mathcal{A}_2(n,k,\delta)$ as above. If $Y_2\neq \emptyset,$ then we have $x_i> x_j $ for all $i\in Z_1$ and $j\in Y_2.$
\end{lemma}

\begin{proof}
By contradiction,  assume that there exist some $i\in Z_1$ and $j\in Y_2$ such that $x_i\leq x_j.$ For $k\in Y$ and $j\nsim k,$ we define a new graph $G'\in \mathcal{A}_2(n,k,\delta)$ by removing the edge $i k$ and adding a new edge $j k.$
Since
$$
\langle Q(G')x,x\rangle-\langle Q(G)x,x\rangle=(x_j-x_i)(x_i+x_j+2x_k)\geq 0,
$$
we get $q(G')\geq q(G)$ and the induced graph $G'[Y]$ has more edges than $G[Y],$ which contradicts the choice of $G.$ The result follows.
\end{proof}

%
%

\begin{lemma}\label{lemma36}
Assume $G\in \mathcal{A}_2(n,k,\delta)$ as defined above.
\begin{enumerate}[(1)]

\item If $Z_2\neq \emptyset,$ then we have $x_i>x_j $ for all $i\in Z_1$ and $j\in Z_2.$

\item If $Y_1,Y_2\neq \emptyset,$ then we have $x_i>x_j$ for any $i\in Y_1$ and $j\in Y_2.$

\item If $Y_1\neq \emptyset,$ then we have $x_i>x_j$ for any $i\in Y_1$ and $j\in Z_1.$
\end{enumerate}
\end{lemma}

\begin{proof}
(1). Using Lemma \ref{lemma21}, we have
$$(q(G)-d(i))(x_{i}-x_{j})=(d(i)-d(j))x_{j}+\sum_{k\in N(i)\setminus N(j)}x_{k}-\sum_{l\in N(j)\setminus N(i)}x_{l}.$$
For each $i\in Z_1$ and $j\in Z_2 ,$ note that $N(j)\setminus \{i\}\subset N(i)\setminus \{j\}.$ Rearranging the last equation, we obtain
\begin{eqnarray}\label{equaion4}
(q(G)-d(i)+1)(x_{i}-x_{j})=(d(i)-d(j))x_{j}+\sum_{k\in N(i)\setminus N[j]}x_{k}.
\end{eqnarray}
As $d(i)>d(j),$ the proof is completed.

(2). If $Y_1,Y_2\neq \emptyset,$ using Lemma \ref{lemma21}, we have
$$(q(G)-d(i))(x_{i}-x_{j})=(d(i)-d(j))x_{j}+\sum_{k\in N(i)\setminus N(j)}x_{k}-\sum_{l\in N(j)\setminus N(i)}x_{l}.$$
For each $i\in Y_1$ and $j\in Y_2 ,$ note that $N(j)\setminus \{i\}\subset N(i)\setminus \{j\}.$  Rearranging the last equation, we obtain
\begin{eqnarray*}\label{equaion5}
(q(G)-d(i)+1)(x_{i}-x_{j})=(d(i)-d(j))x_{j}+\sum_{k\in N(i)\setminus N[j]}x_{k}.
\end{eqnarray*}
as $d(i)>d(j),$ the proof is completed.

(3).
If $Y_1\neq \emptyset,$ applying Lemma \ref{lemma21}, we have
$$(q(G)-d(i))(x_{i}-x_{j})=(d(i)-d(j))x_{j}+\sum_{k\in N(i)\setminus N(j)}x_{k}-\sum_{l\in N(j)\setminus N(i)}x_{l}.$$
For each $i\in Y_1$ and $j\in Z_1 ,$ note that $N(j)\setminus \{i\}\subset N(i)\setminus \{j\}.$ Rearranging the last equation, we obtain
\begin{eqnarray*}\label{equaion6}
(q(G)-d(i)+1)(x_{i}-x_{j})=(d(i)-d(j))x_{j}+\sum_{k\in N(i)\setminus N[j]}x_{k}.
\end{eqnarray*}
as $d(i)>d(j),$ the proof is completed.
\end{proof}
From above,   bearing in mind that $Z_1\neq \emptyset,$   we have

\begin{enumerate}[(1)]

\item If $Y_1= \emptyset,$ obviously we have $Y_2\neq \emptyset$ since $G$ is a connected graph. We divide it into two subcases:

If $Z_2= \emptyset$, from above we know $Z_1\neq \emptyset.$  We easily have that $\max\limits_{i\in V(G)}x_i=\max\limits_{i\in Z_1}x_i$ from Lemma \ref{lemma34};

If $Z_2\neq \emptyset$, we already know that $Z_1\neq \emptyset,$ so we still have $\max\limits_{i\in V(G)}x_i=\max\limits_{i\in Z_1}x_i$ from Lemmas \ref{lemma34} and \ref{lemma36}.

\item If $Y_1\neq \emptyset,$  we divide it into two subcases:

If $Y_2=\emptyset$, we then have $Z_2\neq \emptyset$ for $G\in \mathcal{A}_2(n,k,\delta).$ And $Z_1\neq \emptyset,$ then we   get that $\max\limits_{i\in V(G)}x_i=\max\limits_{i\in Y_1}x_i$ from Lemma \ref{lemma36};

If $Y_2\neq \emptyset$, if $Z_2= \emptyset,$  we  have $\max\limits_{i\in V(G)}x_i=\max\limits_{i\in Y_1}x_i$ from Lemmas \ref{lemma34} and \ref{lemma36}.
 If $Z_2\neq \emptyset,$  we still have $\max\limits_{i\in V(G)}x_i=\max\limits_{i\in Y_1}x_i$ from Lemmas \ref{lemma34} and \ref{lemma36}.
\end{enumerate}
The key step for proving Lemma \ref{lemma38} is to show Lemma \ref{lemma37}.
\begin{lemma}\label{lemma37}
Assume $G\in \mathcal{A}_2(n,k,\delta)$ as defined above. We have
$$\max\limits_{i\in V(G)}x_i-\min\limits_{j\in Y\cup Z}x_j \leq \frac{(\delta-k+2)(k+3)+4}{2(q(G)-n+1)}.$$
\end{lemma}
\begin{proof}
We distinguish the proof into two cases.

\textbf{Case 1}:
    $Y_1=\emptyset.$ Notice that $\max\limits_{i\in V(G)}x_i $ is attained at the vertices in $Z_1$. For each $ i \in Z_1,$ $d(i)=n-\delta+k-3,$ hence the vertex  $i$ is adjacent to all other vertices in $Y\cup Z.$
\medskip

\textbf{Subcase 1.1}:
    If $j\in Z_2,$ we have $N(i)\setminus N[j]=\{k:k\in Y\cup Z$ and $k\nsim j\}.$ Thus we have
    $ d(i)-d(j)\leq \left \lfloor \frac{(\delta-k+2)(k-1)}{4} \right \rfloor+1 $ and $|N(i)\setminus N[j]|\leq \left\lfloor \frac{(\delta-k+2)(k-1)}{4} \right \rfloor+1.$
    Note that $N(j)\setminus \{i\}\subset N(i)\setminus \{j\},$
    applying equation (\ref{equaion4}), we obtain
    \begin{eqnarray*}
    (q(G)-d(i)+1)(x_{i}-x_{j})&=&(d(i)-d(j))x_{j}+\sum_{k\in N(i)\setminus N[j]}x_{k} \\
    &\leq & \left\lfloor \frac{(\delta-k+2)(k-1)}{2} \right\rfloor+2.
    \end{eqnarray*}
    Since $i\in Z_1, \delta \geq k,$ we  have
 $$
    x_{i}-x_{j}  \leq  \frac{\frac{(\delta-k+2)(k-1)}{2}+2}{q(G)-(n-\delta+k-3)+1}
     = \frac{(\delta-k+2)(k-1)+4}{2(q(G)-n+\delta-k+4)}
    < \frac{(\delta-k+2)(k+3)+4}{2(q(G)-n+1)}.
$$

\textbf{Subcase 1.2}:
   If $j\in Y_2,$ we have $N(i)\setminus (N[j])=\{k:k\in Y\cup Z $ and $ k\nsim j\},$ and $N(j)\setminus N[i]=X.$  Thus
   $|N(i)\setminus N[j]|\leq \left\lfloor \frac{(\delta-k+2)(k-1)}{4} \right\rfloor+1.$  Meanwhile, note that $|d(i)-d(j)|\leq \left\lfloor \frac{(\delta-k+2)(k-1)}{4} \right\rfloor+\delta-k+3.$
    Similarly, we obtain
    \begin{eqnarray*}
    (q(G)-d(i)+1)(x_{i}-x_{j})&=&(d(i)-d(j))x_{j}+\sum_{k\in N(i)\setminus N[j]}x_{k}-\sum_{l\in X}x_{l} \\
    &\leq & (d(i)-d(j))x_{j}+\sum_{k\in N(i)\setminus N[j]}x_{k} \\
    &\leq &  \frac{(\delta-k+2)(k-1)}{2}+\delta-k+2+2 \\
    &=& \frac{(\delta-k+2)(k+1)+4}{2}.
    \end{eqnarray*}

    Since $i \in Z_1, \delta \geq k,$ we easily have
    \begin{eqnarray*}
    x_{i}-x_{j} & \leq & \frac{\frac{(\delta-k+2)(k+1)}{2}+2}{q(G)-(n-\delta+k-3)+1}
    = \frac{(\delta-k+2)(k+1)+4}{2(q(G)-n+\delta-k+4)} \\
    &<& \frac{(\delta-k+2)(k+3)+4}{2(q(G)-n+1)}.
    \end{eqnarray*}

\textbf{Case 2}:
    $Y_1\neq\emptyset.$ Notice that $\max\limits_{i\in V(G)}x_i $ is attained at  the vertices in $Y_1$. When $ i \in Y_1,$ $d(i)=n-1,$ hence the vertex  $i$ is adjacent to all other vertices in $V(G).$
\medskip

\textbf{Subcase 2.1}:
    If $j\in Y_2,$ we have $N(i)\setminus N[j]=\{k:k\in Y\cup Z \mbox{ and } k\nsim j\}.$ Thus we have
    $ d(i)-d(j)\leq \left\lfloor \frac{(\delta-k+2)(k-1)}{4} \right\rfloor+1 $ and $|N(i)\setminus N[j]|\leq \left\lfloor \frac{(\delta-k+2)(k-1)}{4} \right\rfloor+1.$
    Note that $N(j)\setminus \{i\}\subset N(i)\setminus \{j\},$
    applying equation (\ref{equaion4}), we obtain
    \begin{eqnarray*}
    (q(G)-d(i)+1)(x_{i}-x_{j})&=&(d(i)-d(j))x_{j}+\sum_{k\in N(i)\setminus N[j]}x_{k} \\
    &\leq & \left\lfloor \frac{(\delta-k+2)(k-1)}{2} \right\rfloor+2.
    \end{eqnarray*}
    Since $i\in Y_1, \delta \geq k,$ we easily have
    \begin{eqnarray*}
    x_{i}-x_{j}  &\leq&  \frac{\frac{(\delta-k+2)(k-1)}{2}+2}{q(G)-(n-1)+1}
    = \frac{(\delta-k+2)(k-1)+4}{2(q(G)-n+2)} \\
    &<& \frac{(\delta-k+2)(k+3)+4}{2(q(G)-n+1)}.
    \end{eqnarray*}

\textbf{Subcase 2.2}:
   If $j\in Z_1,$ we have $N(i)\setminus N[j]=X-\{j\},$ and $N(j)\setminus N[i]=\emptyset.$  Thus
   $|N(i)\setminus N[j]|= (n-1)-(n-\delta+k-3+1)=\delta-k+1.$  Note that $d(i)-d(j)= \delta-k+2$,
   we similarly obtain
    \begin{eqnarray*}
    (q(G)-d(i)+1)(x_{i}-x_{j})&=&(d(i)-d(j))x_{j}+\sum_{k\in N(i)\setminus N[j]}x_{k}
    \\
    &\leq &  \delta-k+2+\delta-k+1 \\
    &< & 2(\delta-k+2).
    \end{eqnarray*}
    Since $i \in Y_1, \delta \geq k,$ we   have
    \begin{eqnarray*}
    x_{i}-x_{j}  <  \frac{2(\delta-k+2)}{q(G)-(n-1)+1} &<& \frac{4(\delta-k+2)}{2(q(G)-n+1)} <  \frac{(\delta-k+2)(k+3)+4}{2(q(G)-n+1)}.
    \end{eqnarray*}

\textbf{Subcase 2.3}:
   If $j\in Z_2,$ we have $N(i)\setminus N[j]=\{k:k\in X\cup Y\cup Z $ and $ k\nsim j\},$ and $N(j)\setminus N[i]=\emptyset.$  Thus we have
   $|N(i)\setminus N[j]|\leq \delta-k+2+ \left\lfloor \frac{(\delta-k+2)(k-1)}{4} \right \rfloor+1.$  Meanwhile, note that $d(i)-d(j)\leq \left\lfloor \frac{(\delta-k+2)(k-1)}{4} \right\rfloor+\delta-k+3.$
    Similarly, we obtain
    \begin{eqnarray*}
    (q(G)-d(i)+1)(x_{i}-x_{j})&=&(d(i)-d(j))x_{j}+\sum_{k\in N(i)\setminus N[j]}x_{k}
    \\
    &\leq &  \frac{(\delta-k+2)(k-1)}{2}+2(\delta-k+2)+2 \\
    &=& \frac{(\delta-k+2)(k+3)+4}{2}.
    \end{eqnarray*}
    Since $i \in Y_1, \delta \geq k,$ we   have
    \begin{eqnarray*}
    x_{i}-x_{j}  \leq  \frac{\frac{(\delta-k+2)(k+3)}{2}+2}{q(G)-(n-1)+1} &=& \frac{(\delta-k+2)(k+3)+4}{2(q(G)-n+2)}
    < \frac{(\delta-k+2)(k+3)+4}{2(q(G)-n+1)}.
    \end{eqnarray*}
The proof is completed.
\end{proof}

Now we can prove Lemma \ref{lemma38}, which is crucial for Theorem \ref{th15}.

\begin{lemma}\label{lemma38}
Let $G$ be a connected graph of order $n\geq F(k,\delta)$ and  minimum degree $\delta \geq k \geq 3.$
For each graph $G\in \mathcal{A}_2(n,k,\delta),$
we have $q(G)< 2(n-\delta+k-3).$
\end{lemma}

\begin{proof}
We assume $G\in  \mathcal{A}_2(n,k,\delta)$ such that $G$ has the largest $Q$-index $q(G)$ among $\mathcal{A}_2(n,k,\delta)$ and $G[Y]$ contains the largest number of edges.
Let $x$ be the eigenvector corresponding to $q(G),$ and $G'[X]$ be the complete graph $K_{\delta-k+2}$ induced by $X.$ Lemmas \ref{lemma33} and \ref{lemma37}   imply that
    \begin{eqnarray*}
    &&\langle Q(G)x,x\rangle-\langle Q(\overline{K_{\delta-k+2}}+K_{n-\delta+k-2})x,x\rangle \\
    &=& \sum_{i,j\in E(G'[X])}(x_i+x_j)^2+\sum_{i\in X,j\in Y}(x_i+x_j)^2-\sum_{{i,j}\in E'}(x_i+x_j)^2 \\
    &\leq & \frac{(\delta-k+2)(\delta-k+1)}{2} \left(2\frac{k-1}{q(G)-(2\delta-k+1)}\right)^2 \\
    && +(k-1)(\delta-k+2) \left(1+\frac{k-1}{q(G)-(2\delta-k+1)}\right)^2 \\
    && -4|E'|\left(1-\frac{(\delta-k+2)(k+3)+4}{2(q(G)-n+1)}\right)^2.
    \end{eqnarray*}
    As $|E'|= \left \lfloor\frac {(\delta-k+2)(k-1)}{4}\right \rfloor+1 \geq \frac{(\delta-k+2)(k-1)+1}{4},$ we have
    \begin{eqnarray*}
    &&\langle Q(G)x,x\rangle-\langle Q(\overline{K_{\delta-k+2}}+K_{n-\delta+k-2})x,x\rangle \\
    &\leq & \frac{(\delta-k+2)(\delta-k+1)}{2} \left(2\frac{k-1}{q(G)-(2\delta-k+1)}\right)^2 \\
    &&+(k-1)(\delta-k+2)\left(1+\frac{k-1}{q(G)-(2\delta-k+1)}\right)^2 \\
    &&-\left((k-1)(\delta-k+2)+1\right)\left(1-\frac{(\delta-k+2)(k+3)+4}{2(q(G)-n+1)}\right)^2.
    \end{eqnarray*}
    Since $n\geq F(k,\delta),$ and $q(G)>2(n-\delta+k-3)-1$ by Lemma \ref{lemma32}, we have
    $$\langle Q(G)x,x\rangle-\langle Q(\overline{K_{\delta-k+2}}+K_{n-\delta+k-2})x,x\rangle <0.$$
    According to the Rayleigh's principle,
    $$\frac{\langle Q(\overline{K_{\delta-k+2}}+K_{n-\delta+k-2})x,x\rangle}{\langle x,x\rangle}\leq q(\overline{K_{\delta-k+2}}+K_{n-\delta+k-2})=2(n-\delta+k-3).$$
    Therefore, we have $q(G)=\frac{\langle Q(G)x,x\rangle}{\langle x,x\rangle}<2(n-\delta+k-3).$
\end{proof}
Now we are ready to prove Theorem \ref{th15}.

\begin{proof}
By Lemma \ref{lemma22}, we have
$$2(n-\delta+k-3)\leq q(G) \leq \frac{2m}{n-1}+n-2. $$
Therefore
    \begin{eqnarray*}
    m& \geq & \frac{(n-2\delta+2k-4)(n-1)}{2} \\
    &=& \frac{n(n-1)}{2}-(\delta-k+3)(n-\delta-2)+n-\delta-2-(\delta-k+2)(\delta+1) \\
    &>& \frac{n(n-1)}{2}-(\delta-k+3)(n-\delta-2),
    \end{eqnarray*}
the last inequality holds as $n\geq F(k,\delta).$
By Lemma \ref{lemma23}, $G$ is $k$-connected unless $G \in {A}(n,k,\delta).$ Together with Lemmas~ \ref{lemma31} and \ref{lemma38}, the result follows.
\end{proof}


\end{document}